\documentclass[reqno]{amsart}
\usepackage{latexsym}
\usepackage{amssymb}
\usepackage{amsmath}
\usepackage{amsfonts}
\usepackage{amssymb}
\usepackage{mathtools}
\usepackage{amsthm}
\usepackage{amsbsy}
\usepackage{amscd}
\usepackage{bm}
\usepackage[all]{xy}
\usepackage{epigraph}
\usepackage{xcolor}

\usepackage{color}
 
\newcommand{\set}[2]{\lbrace#1~:~#2\rbrace}
\newcommand{\dom}{\mathbf{dom}}
\newcommand{\lr}[1]{\langle#1\rangle}
\newcommand{\define}{\stackrel{\text{df}}{\Longleftrightarrow}}
\newcommand{\adopts}{\rightsquigarrow}
\newcommand{\adoptsM}{\adopts_{\mat{M}}}

\newcommand{\CF}{\mathbf{C}_{F}}
\newcommand{\CFs}{\mathbf{C}_{F^{\ast}}}

\newcommand{\C}{\mathbf{C}}
\newcommand{\Cn}{\mathbf{Cn}}

\newcommand{\CmatMrs}{\C_{\mat{M}}^{r\ast}}

\newcommand{\CcalMrs}{\C_{\mathcal{M}}^{r\ast}}

\newcommand{\modelM}{\models_{\mat{M}}}
\newcommand{\modelcalM}{\models_{\mathcal{M}}}
\newcommand{\modelMr}{\models^{r}_{\mat{M}}}
\newcommand{\modelcalMr}{\models^{r}_{\mathcal{M}}}
\newcommand{\modelMrs}{\models^{r\ast}_{\mat{M}}}
\newcommand{\modelcalMrs}{\models^{r\ast}_{\mathcal{M}}}

\newcommand{\snakeF}{\mid\!\sim_{F}}
\newcommand{\snakeFs}{\mid\!\sim_{F^{\ast}}}

\newcommand{\vdashC}{\vdash_{\C}}
\newcommand{\vdashL}[1]{\vdash_{#1}}
\newcommand{\VdashL}[1]{\Vdash_{#1}}
\newcommand{\sequentL}[1]{\!\Rightarrow_{#1}\!}

\newcommand{\Cl}{\textsf{Cl}}
\newcommand{\Int}{\textsf{Int}}

\newcommand{\Lan}{\mathcal{L}}
\newcommand{\LanF}{\mathcal{L}^{\circ}}

\newcommand{\V}{\mathcal{V}}
\newcommand{\Vs}[1]{\mathcal{V}^{\ast}_{\mat{#1}}}
\newcommand{\setCM}[1]{#1_{\mathbf{c}_{\mat{M}}}}
\newcommand{\Var}{\mathcal{V}}
\newcommand{\VarL}{\mathcal{V}_{\mathcal{L}}}
\newcommand{\VarF}{\mathcal{V}_{\mathcal{L}^{\circ}}}

\newcommand{\AtomL}{\mathcal{A}_{\Lan}}

\newcommand{\FuncL}{\mathcal{F}_{\Lan}}

\newcommand{\FormsL}{\mathbf{Fm}_{\mathcal{L}}}

\newcommand{\FormAL}{\mathfrak{F}_{\mathcal{L}}}

\newcommand{\alg}[1]{\mathbf{#1}}
\newcommand{\mat}[1]{\mathbf{#1}}
\newcommand{\booleTwo}{\mat{B}_{2}}

\newcommand{\LT}[1]{\mathbf{LT}_{#1}}

\newcommand{\ExtInt}{\mathbf{ExtInt}}

\newcommand{\one}{\boldsymbol{1}}
\newcommand{\zer}{\boldsymbol{0}} 

\newtheorem{examp}{Example}[section]
\newtheorem{conjecture}{Conjecture}

\newtheorem{lem}{Lemma}[section]
\newtheorem{cor}{Corollary}[section]

\newtheorem{prop}{Proposition}[section]
\newtheorem{defn}{Definition}[section]

\begin{document}

\title[Two Modes of Nonmonotonic Consequence]{Two Modes of Nonmonotonic Consequence$^\dagger$}

\author[Alexei Muravitsky]{Alexei Muravitsky}

\address{Louisiana Scholars' College, \\
	Northwestern State University\\
	Natchitoches, LA 71497, USA.}
\email{alexeim@nsula.edu}

\thanks{$^\dagger$This is the text of my speech at the Logica Universalis webinar, which took place on May 11, 2022.}

\ 

\begin{abstract}
We discuss two ways to implement a semantic approach to nonmonotonic consequence relations in an arbitrary propositional language.
For one particular language, we also discuss the proof-theoretic framework that we connect with this semantic approach.
This article is an addition to~\cite{mur2021}.
\end{abstract}

\maketitle

\section{What do we know and how do we know about nonmonotonic consequences?}\label{section:one}
We discuss non-monotony expressed in formalized languages. All main objective languages that will be involved are propositional.\footnote{First-order language is used only as an auxiliary tool.} 

Aristotle is credited with distinguishing between \emph{what we know} and \emph{how we know}. G. Kreisel who made this remark in~\cite{kreisel-1}, p. 42., adds,
\begin{quote}
	The point I wish to stress, again as a matter of accumulated experience, is how much \emph{easier} the former question is to answer (to our satisfaction) than the latter. Of the \emph{what}, whose aspects strike the (mind's) eye, we have easily accessible practical knowledge (which constitutes the threshold for additions by further theory). We also would like to know about the \emph{how}, but the price is high. (Ibid)
\end{quote}

\vskip 0.1in
We note that it is not only the possible easiness of answering the \emph{what} question that prompts us to start with it but also the view that ``knowing how we know is one small department of knowing what we know'' (Russell).\footnote{Here is a full quote:
\begin{quotation}
	I reverse the process which has been common in philosophy since Kant. It has been common among philosophers to begin with how we know and proceed afterwards to what we know. I think this is a mistake, because knowing how we know is one small department of knowing what we know. (\cite{russell1959}, chapter II)
\end{quotation}}

Speaking of ``easily accessible practical knowledge'', Kreisel probably had in mind, first of all, modern mathematics and formal logic. As is known, both are based on monotonic reasoning. On the contrary, with respect to nonmonotonic reasoning, it seems that answering both \emph{what} and \emph{how} questions is not an easy task. We do not think that the following statement by D. Makinson can be taken seriously.
\begin{quote}
	Of course, human have been reasoning nonmonotonically for as long as they have been reasoning at all. (\cite{makinson1994}, section 1.1)
\end{quote}

Despite Makinson's remark, it is precisely because of the lack of rich \emph{contextual practices} involving nonmonotonic reasoning that it is so difficult to answer \emph{what we know} and \emph{how we know} about it.

However, non-monotonicity does occur in colloquial speech. R. Stalnaker~\cite{stalnaker1994} gives an example of nonmonotonic argument based on Grice's concept of (conversational) `implicature'.\footnote{See~\cite{grice1989}, part I.} This example is remarkable only in one respect --- it indicates the need to distinguish between the concepts of argumentation and logic. Logic deals exclusively with forms, while content can be involved in argumentation. An example of the latter is that of Stalnaker --- the nonmonotonic effect of his example is based on the interaction of contents, more precisely, between implicit and explicit contents. Understanding logic as a science dealing with forms of judgment, we will pay attention to sentential variables and, as a consequence, to the role of  substitution, understood both as an operation and as a rule of inference. But first, we outline the \emph{monotonicity-non-monotonicity} dichotomy as such.

Before proceeding to a formal presentation, we note that we are using the concept of logic in its dialectical, and not demonstrative, manner. That is, we will consider the relationship between a set of premises and the conclusions that can be drawn from these premises according to some standards.

For this,  we fix a propositional language, $\Lan$, (later we will be more specific), whose formulas are denoted by letters $ \alpha $, $ \beta $, etc.; sets of formulas by $ X $, $ Y $, etc.; and the set of all formulas by $ \FormsL $.
Given sets $X$ and $Y$ of formulas, we write $X\Subset Y$ if $X$ is a finite subset of $Y$.

The operation of (uniform or simultaneous) substitution is understood in the usual way. We denote substitutions in $\Lan$ by $\sigma$, $\delta$, $\ldots$ (with or without subscripts or other marks). The result of substitution $\sigma$ in a formula $\alpha$ is denoted by $\sigma(\alpha)$; accordingly, 
\[
\sigma(X):=\set{\sigma(\alpha)}{\alpha\in X}.
\]

Let $\C$ be a mapping $\C:\mathcal{P}(\FormsL)\longrightarrow\mathcal{P}(\FormsL)$. Given a set $X\subseteq\FormsL$, $\C(X)$ is understood as the set of all conclusions obtained from $X$. 

The definition 
\begin{equation}\label{E:relation-operator}
	X\vdashC\alpha~\define~\alpha\in\C(X)
\end{equation}
will be very useful subsequently. 

The last definition implies conversion; namely, given a relation $\vdash\subseteq\mathcal{P}(\FormsL)\times\FormsL$, we define:
\begin{equation}\label{E:operator-relation}
\alpha\in\C_{\vdash}(X)~\define~X\vdash\alpha.	
\end{equation}

The definitions~\eqref{E:relation-operator} and~\eqref{E:operator-relation} allow to use relations of the above type and the corresponding operators interchangeably.

The following properties of the $\C$ operator have long been in the spotlight.\footnote{See e.g.~\cite{makinson2005b}.}

\begin{itemize}
	\item [(con-1)] $X\subseteq \C(X)$;\quad(\emph{reflexivity})
	\item [(con-2)] $X\subseteq Y$ implies $\C(X)\subseteq\C(Y)$;\quad(\emph{monotonicity})
	\item [(con-3)] $\C(\C(X))\subseteq\C(X)$;\quad(\emph{closedness})
		\item [(con-4)] $\C(X)\subseteq\bigcup\set{\C(Y)}{Y\Subset X}$;\quad(\emph{finitariness} or \emph{compactness})
	\item [(con-5)] $X\subseteq\C(Y)$ implies $\C(X)\subseteq\C(Y)$;\quad(\emph{cumulative transitivity})
	\item [(con-6)] $X\subseteq Y\subseteq\C(X)$ implies $\C(X)\subseteq\C(Y)$;\quad(\emph{weak cumulative transitivity}\footnote{This property is also known as \emph{cautious monotony}.})
	\item [(con-7)] $X\subseteq\C(Y)$ implies $\C(X\cup Y)\subseteq\C(Y)$;\quad(\emph{strong cumulative transitivity})
	\item [(con-8)] $\bigcup\set{\C(Y)}{Y\Subset X}\subseteq\C(X)$;\quad(\emph{finitary inclusion})
	\item [(con-9)] if $X\neq\varnothing$, then $\C(X)\subseteq\bigcup\set{\C(Y)}{Y\neq\varnothing~\text{and}~Y\Subset X}$.\footnote{We eliminate curly brackets for singleton premises; thus $\C(\alpha)$ stands for $\C(\{\alpha\}$).}\quad\\(\emph{strong finitariness})
\end{itemize}

A. Tarski, who introduced the operator $\C$ into the discourse in his address to the Polish Mathematical Society in 1928,\footnote{He used the notation `$\Cn$' instead of `$\C$'.} began with the following remark.
\begin{quote}
	Our object in this communication is to define the meaning, and to establish the elementary properties, of some important concepts belonging to the	\emph{methodology of deductive sciences}, which, following Hilbert, it is customary to call \emph{metamathematics}. (\cite{tarski1930a})
\end{quote}

\vskip 0.1in
In the published version of this address, Tarski formulated (con-1)--(con-3). Given a set $X\subseteq\FormsL$, he called $\C(X)$ ``the \emph{consequences of the set} $X$,'' which suggests to call any operator $\C$ satisfying these properties a \emph{consequence operator}. It was the rich variety of contextual practices in mathematics and the deductive sciences that prompted him to introduce this notion. In his own worlds,
\begin{quote}
	The concept of \emph{logical consequence} is one of those whose introduction into the field of strict formal investigation was not a matter of arbitrary decision on the part of this or that investigator; $\ldots$ (\cite{tarski1936b})
\end{quote} 

 Tarski also considered (con-4) but not other properties. One interesting property was introduced over a quarter of a century later in~\cite{los-suszko1958}. Namely,
 \vskip 0.1in
 \begin{itemize}
 	\item [(con-10)] $\alpha\in\C(X)$ implies $\sigma(\alpha)\in\C(\sigma(X))$. (\emph{structurality}) 
 \end{itemize}

As for the \emph{what / how} questions about monotonic reasoning, at least with regard to its forms as presented in mathematics and the deductive sciences, it is generally agreed that the properties (con-1)--(con-3) seem to answer the first question, or at least can be regarded as ``the threshold for additions by further theory'' (Kreisel).

Frege and Tarski can be credited for pioneering to answer the second question. In particular, Tarski pointed out the following two ways of implementing monotonic reasoning in metamathematics. According to him, the first way is as follows.
\begin{quote}
	From the sentences of any set $X$ certain other sentences can be obtained by means of certain operations called \emph{rules of inference}. (\cite{tarski1930a})
\end{quote}

The second way is through semantics.\footnote{Cf.~\cite{tarski1936b}.} In the most general terms, one can imagine a non-linguistic machinery $\mathfrak{M}$ such that it is natural to say that the sentence $\alpha$ follows from the set of premises $X$, when $\mathfrak{M}$ accepts $\alpha$ whenever $\mathfrak{M}$ accepts all premises of $X$. For propositional languages, this leads to matrix consequences.\footnote{See e.g.~\cite{wojcicki1988}, chapter 3, or~\cite{citkin-muravitsky2022}, chapter 4.} And although from this level of abstraction the implementation of this idea may vary, we will not discuss here all the alternatives.\footnote{See~\cite{citkin-muravitsky2022}, section 4.3.3.}

Unfortunately, at this point, little can be said about answers to the \emph{what} question in relation to nonmonotonic reasoning. To reiterate, we argue that it is the lack of contextual practices of nonmonotonic reasoning that has led to speculating about ``the minimum conditions that a consequence relation should satisfy to represent a bona fide nonmonotonic \emph{logic} $\ldots$''\footnote{Cf.~\cite{kraus_et_al1990}, section 1.2; comp.~\cite{gabbay1985}.} In contrast to this optimistic hope, one can see a growing confusion in the endeavor of  trying to learn something about the objective world of reasoning through pure thinking. Another confusion, the \emph{what-how} confusion, occurs when someone, in an attempt to give a general definition of nonmonotonic reasoning, gives examples of answers to the \emph{how} question.

If contextual practices of nonmonotonic reasoning are lacking, what can the theory of nonmonotonic reasoning be based on? One possible answer is that it can be based on conceptualized practices of monotonic reasoning. Recall that algebraic logic began with George Boole's algebraization of Aristotle's syllogistic. And to give credit to Kreisel's insight, to answer the \emph{what} question, we should simply propose a $\C$ operator that would not satisfy (con-2).
However, there are many candidates for answering the \emph{how} question, not all of them are indisputable. What seems certain is that any mode of nonmonotonic reasoning worth exploring should be considered together with its monotonic counterpart.\footnote{This idea was systematically implemented in~\cite{makinson2005b}, although only as an explanatory tool, and not as a methodological guide.} Following this observation, we introduce the notation `$\Cn$' to indicate that the last operator satisfies the properties (con-1)--(con-3). The seeming contradiction with the last paragraph is resolved by establishing relations between the operators $\Cn$ and $\C$.

\section{Logicality of (logical) consequences}\label{section:two}
First of all, we must explain the title of this section.

Even in the case of monotonic reasoning, if the answer to the 
\emph{what} question can be explicated in the form of conditions (con-1)--(con-3), any possible answer to the \emph{how} question is fraught with danger. Here are two views on the issue. The first comes from the corner of constructive philosophy.

\vskip 0.1in
\begin{quote}
	A logical inference is a movement from certain propositions (the premises) to a further proposition (the conclusion). We still need to know, however, which of these movements are to be called logical. (\cite{lorenzen1987}, section 4)
\end{quote}

\vskip 0.1in
The second point of view sounds from the corner of proof theory.

\vskip 0.1in
\begin{quote}
	If we simply axiomatize or define the notion of logical consequence with the understanding that a logical consequence holds when this follows logically (!) from the axioms or the definitions, then one may rightly say whether anything really is achieved. (\cite{prawitz1974}, section 2)
\end{quote}

\vskip 0.1in
Such sentiments sometimes, though rarely, appeared in the camp of researchers of nonmonotonic reasoning; in particular, in the second sentence of the following quote (if you forgive the \emph{what-how} confusion in the first).

\vskip 0.1in
\begin{quote}
	$[\ldots]$ what is wanted is not a specific technical definition for a specific non-monotonic logic, but a general account of what consequence is supposed to mean---of what concept some specific technical definition is trying to capture. And one should expect such an account to explain why it is reasonable to call the concept by the name ``consequence''. (\cite{stalnaker1993})
\end{quote}

\vskip 0.1in
Apparently, D. Makinson~\cite{makinson1994}, section 2.2, was the first to propose a clear criterion for the logicality of a nonmonotonic consequence; namely, he formulated that for two operators $\Cn$ (monotonic) and $\C$ (presumably nonmonotonic), the operator $\C$ is logical (with respect to $\Cn$) if
\begin{equation}\label{E:logical-according-Makinson}
	\Cn\C=\C=\C\Cn.
\end{equation}

It should be clarified that any standard on the logicality of nonmonotonic reasoning is an answer to the \emph{what} question, not to the \emph{how} question. In formulating such a standard, we want to prescribe the conditions of what we expect it to be.

To facilitate the analysis of~\eqref{E:logical-according-Makinson}, we introduce two relations similar to~\eqref{E:relation-operator}. So we define:
\begin{equation}\label{E:vdash-Cn-equivalence}
	X\vdash\alpha~\define~\alpha\in\Cn(X),
\end{equation}
and
\begin{equation}\label{E:snake-C-equivalece}
X\mid\!\sim\alpha~\define~\alpha\in\C(X).
\end{equation}

We find that~\eqref{E:logical-according-Makinson} is too strong. Let us consider first the inequality
\begin{equation}\label{E:logical-Makinson-inequality}
	\Cn\C\le\C,
\end{equation}
and express it as follows:
\[
\set{\beta}{X\mid\!\sim \beta}\vdash\alpha~\Longrightarrow~X\mid\!\sim \alpha,
\]
or in a shorter form
\[
\C(X)\vdash\alpha~\Longrightarrow~X\mid\!\sim\alpha.
\]

The last conditional does not look convincing; it is rather unjustifiably specific.\footnote{Although~\eqref{E:logical-Makinson-inequality}, and even~\eqref{E:logical-according-Makinson}, is fulfilled in any epsilon inference probabilistic operator (cf.~\cite{makinson1994}, observation 3.5.2), Makinson, however, makes the following remark.
	\begin{quote}
		It must be said, however, that there is a rather unsatisfying gap between the initial intuitions behind the epsilon approach and their technical formulation [$\dots$].
\end{quote}} 

Turning to the inequality
\[
\C\le\C\Cn,
\]
the following generalization seems to be more flexible:
\begin{equation}\label{E:minimum-condition-for-logical-3}
	(X\mid\!\sim\alpha~\text{and}~Y\subseteq\Cn(X))~\Longrightarrow~X\cup Y\mid\!\sim\alpha.
\end{equation}
Besides,~\eqref{E:minimum-condition-for-logical-3} can be seen as an attempt to make the relation $|\!\!\!\sim$ somewhat monotonic. As we will show in Section~\ref{section:M-r-consequence}, in some interesting models,~\eqref{E:minimum-condition-for-logical-3} does not hold. Nonetheless, as will be seen in the sequel, limited monotonicity can be achieved, although on a different path.

We call a nonmonotonic relation $|$\hspace{-0.03in}$\sim$  \textit{\textbf{logical relative to} a given monotonic consequence relation} $\vdash$ if it satisfies the conditions:
\begin{equation}\label{E:minimum-condition-for-logical-1}
	X\vdash\alpha~\Longrightarrow~X\mid\!\sim\alpha,\footnote{See comments about~\eqref{E:minimum-condition-for-logical-1} in~\cite{makinson2005b}, section 1.3.}
\end{equation}
and
\begin{equation}\label{E:minimum-condition-for-logical-2}
	(X\mid\!\sim\alpha~\text{and}~\alpha\vdash\beta)~\Longrightarrow~X\mid\!\sim\beta.
\end{equation}

The conditions~\eqref{E:minimum-condition-for-logical-1} and~\eqref{E:minimum-condition-for-logical-2} can be formulated in terms of operators $\Cn$ and $\C$ that are defined by \eqref{E:vdash-Cn-equivalence} and~\eqref{E:snake-C-equivalece}, respectively, as follows.

We say that an operator $\C$ is \textbf{logical relative to} $\Cn$ if, and only if, the following conditions are satisfied:
\begin{itemize}
	\item [(log-1)] $\Cn\le\C$;
	\item [(log-2)] $\alpha\in\C(X)$ and $\beta\in\Cn(\alpha)$ imply $\beta\in\C(X)$.
\end{itemize}

\vskip 0.1in
Continuing the discussion of the \emph{what} question in relation to nonmonotonic reasoning, let us turn to the map of properties (con-1)--(con-9) presented in the next proposition.
However, first, we will divide these properties into five (colored) groups: the \textcolor{violet}{\emph{reflexivity property} (con-1)},
the \textcolor{brown}{\emph{closedness property} (con-3)}, the \textcolor{red}{\emph{monotonicity property} (con-2)}, the \textcolor{blue}{\emph{cumulativity properties} (con-5)--(con-7)}, and the \textcolor{green}{\emph{finitariness properties}  (con-4)}, \textcolor{green}{(con-8)} and \textcolor{green}{(con-9)}.
\begin{prop}[cf.~\cite{citkin-muravitsky2022}, proposition 4.2.5]\label{P:con-connections}
	The following implications hold:
	\begin{itemize}
		\item [i)]\quad{\em\textcolor{violet}{(con-1)}} and {\em\textcolor{blue}{(con-6)}} imply {\em\textcolor{brown}{(con-3)}};
		\item [ii)]\quad{\em\textcolor{red}{(con-2)}} and {\em\textcolor{brown}{(con-3)}} imply {\em\textcolor{blue}{(con-5)}};
		\item [iii)]\quad{\em\textcolor{blue}{(con-5)}} implies {\em\textcolor{blue}{(con-6)}};
		\item [iv)]\quad{\em\textcolor{blue}{(con-7)}} implies {\em\textcolor{blue}{(con-6)}};
		\item [v)]\quad{\em\textcolor{violet}{(con-1)}} and {\em\textcolor{blue}{(con-6)}} imply {\em\textcolor{blue}{(con-7)}};
		\item [vi)]\quad{\em\textcolor{violet}{(con-1)}} and {\em\textcolor{blue}{(con-7)}} imply {\em\textcolor{brown}{(con-3)}};
		\item [vii)]\quad{\em\textcolor{violet}{(con-1)}} and {\em\textcolor{red}{(con-2)}} and {\em\textcolor{brown}{(con-3)}} imply {\em\textcolor{blue}{(con-7)}};
		\item [viii)]\quad{\em\textcolor{green}{(con-9)}} implies {\em\textcolor{green}{(con-4)}};
		\item [ix)]\quad{\em\textcolor{red}{(con-2)}} and {\em\textcolor{green}{(con-4)}} imply {\em\textcolor{green}{(con-9)}};
		\item [x)]\quad{\em\textcolor{red}{(con-2)}} implies {\em\textcolor{green}{(con-8)}};
		\item [xi)]\quad{\em\textcolor{green}{(con-4)}} and {\em\textcolor{green}{(con-8)}} imply {\em\textcolor{red}{(con-2)}}.\footnote{The implication (xi) was noted by Makinson who wrote:
			\begin{quote}
				The compactness biconditional must [$\dots$] fail in any nonmonotonic logic, and although the right-to-left half of it [$\ldots$] if $A_0\mid\!\sim x$ for some $A_0$ included in $A$, then $A\mid\!\sim x$] seems the more clearly inappropriate, even its converse is rather dubious$\dots$ (\cite{makinson1994}, p. 41)
		\end{quote}}
	\end{itemize}
\end{prop}

As is clear from the last proposition, if we intend to challenge the monotonicity property (con-2), and we do so, we must be aware that other properties will be affected. 

Note that (con-1) is always satisfied if the consequence operator in question is logical. We also observe that the cumulative properties and the finitariness properties do not interact. Further, in the presence of (con-1), if (con-7) does not hold, it is problematic that (con-3) holds.

In the mode of nonmonotonic reasoning that we are about to propose (thereby answering the \emph{how} question), the closedness and cumulative properties that have been the focus of nonmonotonic research\footnote{See e.g.~\cite{makinson1989}.} will fail. Therefore, we will be interested in the finitariness properties. But even in this case, by Proposition~\ref{P:con-connections}-xi, we cannot have both (con-4) and (con-8). As is expected, (con-4) will take precedence.

\section{Matrix (monotonic) consequence and two restricted matrix consequences}\label{section:three}
In this section, we show some ways of passing from monotonic consequences to nonmonotonic consequences. We note that, although our starting point will be language-independent, in rejection of (con-2) we are forced to deal with specifics of the language involved. However, further, in showing which properties of the list (con-1)-(con-9) are preserved, we again can deal with a language-independent framework.

Fixing a language $\Lan$, we also fix the following notations:
\begin{itemize}
	\item $\VarL$ is an infinite set of propositional variables;
	\item $\FuncL$ is a set of \textit{logical constants}, also known as \textit{logical connectives}, including $0$-ary logical constants, also known as \textit{constants};
	\item $\AtomL$ is the set of all \textit{atomic formulas}, that is, the set of all variables and constants;
	\item the universal algebra $\FormAL:=\lr{\FormsL,\FuncL}$ is called the $\Lan$-\textit{formula algebra} or simply \textit{formula algebra} when $\Lan$ is unambiguous;
	\item given a set $X\subseteq\FormsL$, $\V(X)$ is the set of propositional variables occurring in all formulas from $X$. 
\end{itemize}

Interpreting elements of $\FuncL$ as operations on a nonempty set \textsf{A}, we obtain an algebra $\alg{A}:=\lr{\textsf{A},\FuncL}$. Subsequently, we also use the notation $|\alg{A}|:=\textsf{A}$.

A system $\mat{M}:=\lr{\alg{A}, D}$, where $D\subseteq|\alg{A}|$ is called a (\textit{logical}) \textit{matrix}.

Any homomorphism $v:\FormAL\longrightarrow\alg{A}$ is called a \textit{valuation} in $\alg{A}$; if $\alg{A}$ is part of a matrix $\mat{M}$, $v$ is also a \textit{valuation} in $\mat{M}$. The result of the application of $v$ to $\alpha$ is denoted by $v[\alpha]$; and for $X\subseteq\FormsL$,
\[
v[X]:=\set{v[\alpha]}{\alpha\in X}.
\]

We note that any mapping $v:\VarL\longrightarrow\alg{A}$ can be uniquely extended to a homomorphism $\hat{v}:\FormAL\longrightarrow\alg{A}$. Because of this, we often regard $v$ as a \textit{valuation} in $\alg{A}$

\subsection{Matrix consequence}\label{section:three.one}
In defining a monotonic consequence, in order to keep a language-independent framework, we use one of Tarski's modes of semantic consequence (Section~\ref{section:one}).

\vskip 0.1in
Given a matrix $\mat{M}=\lr{\alg{A},D}$, a relation $\modelM\subseteq\mathcal{P}(\FormsL)\times\FormsL$ is defined as follows:
\begin{equation}\label{E:M-consequence}
X\modelM\alpha~\define~\text{for any valuation $v$, $v[X]\subseteq D$ implies $v[\alpha]\in D$}.
\end{equation}

The last relation can be generalized.

Let $\mathcal{M}$ be a nonempty set of matrices. Then we define:
\begin{equation}\label{E:calM-consequence}
X\modelcalM\alpha\define~X\modelM\alpha~\text{for every $\mat{M}\in\mathcal{M}$}.
\end{equation}

We will use the notations $\Cn_{\mat{M}}$ and $\Cn_{\mathcal{M}}$, thereby indicating that the former is the operator corresponding to $\modelM$ and the latter is the operator corresponding to $\modelcalM$, both in the light of definition~\eqref{E:operator-relation}.\footnote{It must be clear that if $\mathcal{M}=\{\mat{M}\}$, $\modelM$ and $\modelcalM$ can be used interchangeably.} These notations are justified by the following proposition.

\begin{prop}\label{P:lukasiewicz-tarski}
	Any operator $\Cn_{\mathcal{M}}
	$ satisfies the conditions {\em(con-1)--(con-3)} and {\em(con-10)}.\footnote{This statement goes back to~\cite{lukasiewicz-tarski1930}, theorem 3; For modern expositions, the reader is offered to consult~\cite{wojcicki1988}, theorem 3.1.3, or~\cite{dunn-hardegree2001}, corollary 6.13.3, or~\cite{citkin-muravitsky2022}, proposition 4.3.11.}
\end{prop}

We note that the next equality is an immediate consequence of~\eqref{E:calM-consequence}.
\begin{equation}\label{E:Cn_calM}
	\Cn_{\mathcal{M}}(X)=\bigcap_{\mat{M}\in\mathcal{M}}\Cn_{\mat{M}}(X).
\end{equation}

\subsection{M-r-consequence and $\mathcal{M}$-r-consequence}\label{section:three.two}
Let $\V\subseteq\VarL$ and $v$ be a valuation in an algebra $\alg{A}$, regarded as a mapping on $\VarL$. We call $v\upharpoonright\V$ a \textit{restriction} (of $v$) to $\V$.
Given $\V\subseteq\VarL$, we denote by $v_{\V}$ a \emph{restricted valuation} with $\dom(v)=\V$, where $\dom(v)$ is the domain of $v$.

Given two restrictions $v$ and $w$ in the same algebra, we write $v\le w$ (or $w\ge v$) and say that $w$ is an \textit{extension} of $v$ if $\dom(v)\subseteq\dom(w)$.

In light of these notions, the relation $\modelM$ (see~\eqref{E:M-consequence}) can be reformulated as follows:
\[
\begin{array}{rcl}
X\modelM\alpha &\Longleftrightarrow &\text{for any restriction $v$ with $\dom(v)=\V(X)$},\\
&&\text{if $v[X]\subseteq D$, then for any extension $w\ge v$}\\
&&\text{with $\V(X\cup\{\alpha\})\subseteq\dom(w)$, $w[\alpha]\in D$}.
\end{array}
\]

Changing the right-hand side of the last equivalence, we obtain the following definition.\footnote{The idea of this definition is due to D. Makinson, who used it with a two-element Boolean matrix; cf.~\cite{makinson2005a, makinson2007}.}
\begin{equation}\label{E:M-r-consequence}
	\begin{array}{rcl}
		X\modelMr\alpha &\define &\text{for any restriction $v$ with $\dom(v)=\V(X)$},\\
		&&\text{if $v[X]\subseteq D$, then there is an extension $w\ge v$}\\
		&&\text{with $\V(X\cup\{\alpha\})\subseteq\dom(w)$, such that $w[\alpha]\in D$}.
	\end{array}
\end{equation}

Generalizing the last definition to any nonempty set $\mathcal{M}$ of matrices, we obtain the following.
\begin{equation}\label{E:calM-r-consequence}
X\modelcalMr\alpha~\Longleftrightarrow~X\modelMr\alpha~\text{for every $\mat{M}\in\mathcal{M}$}.
\end{equation}

The last definitions define M-r-\textit{consequence} and $\mathcal{M}$-r-\textit{consequence}, respectively; we call them collectively \textit{restricted matrix consequences}.

We denote by $\C_{\mat{M}}^{r}$  and $\C_{\mathcal{M}}^{r}$ the operators corresponding to $\modelMr$ and $\modelcalMr$, respectively, and then we observe:
\begin{equation}\label{E:C_calMr}
\C_{\mathcal{M}}^{r}(X)=\bigcap_{\mat{M}\in\mathcal{M}}\C_{\mat{M}}^{r}(X),
\end{equation}
which immediately follows from~\eqref{E:M-r-consequence}

\vskip 0.1in
To prove the next proposition, we borrow a definition from~\cite{mur2021} and then state a lemma.
\begin{defn}[comp.~\cite{mur2021}, definition 3.5]
Let $\mat{M}$ be a matrix.	Given a nonempty set $X$ of formulas and a valuation $v$ in $\mat{M}$, we say that
	$v$ \textbf{adopts} $X$ in $\mat{M}$, in symbols $v\adoptsM X$, if there is an extension $w$ of $v$ with $\dom(w)=\Var(X)\cup\dom(v)$ which validates all formulas of $X$. If $X=\{\alpha\}$, we write $v\adoptsM\alpha$.
\end{defn}

\begin{lem}[cf.~\cite{mur2021}, proposition 3.4]\label{L:one}
		Given $X\cup\{\alpha\}\subseteq\FormsL$, let us denote $\Var_{0}:=\Var(X)\cap\Var(\alpha)$. Then for any matrix $\mat{M}$,
	\begin{equation}\label{E:equivalence-M-r}
		X\modelMr\alpha~\Longleftrightarrow~(\text{for any $v_{\Var_0}$ in $\mat{M}$},~v_{\Var_0}\adoptsM X~\Longrightarrow~v_{\Var_0}\adoptsM\alpha).
	\end{equation}
\end{lem}
\begin{prop}\label{P:Mr-logical}
Any operator $\C_{\mathcal{M}}^{r}$ is logical relative to
$\Cn_{\mathcal{M}}$.
\end{prop}
\begin{proof}
	To prove the statement, we have to verify the conditions (log-1) and (log-2) for $\Cn_{\mathcal{M}}$ and $\C_{\mathcal{M}}^{r}$. It must be clear that for this, it suffices to verify (log-1) and (log-2) for $\Cn_{\mat{M}}$ and $\C_{\mat{M}}^{r}$, where $\mat{M}\in\mathcal{M}$.
	
	Let us fix such $\mat{M}=\lr{\alg{A},D}$.
	
	The condition (log-1), which is $\Cn_{\mat{M}}\le\C_{\mat{M}}^{r}$, obviously is true.
	
To prove (log-2), we use its form~\eqref{E:minimum-condition-for-logical-2}, that is,
\[
X\modelMr\alpha~\text{and}~\alpha\modelM\beta~\text{imply}~X\modelMr\beta.
\]

Intending to apply~\eqref{E:equivalence-M-r}, we denote $\V_{0}=\V(X)\cap\V(\beta)$ and assume that $v_{\V_0}\adoptsM X$.

This implies that there is a restricted valuation $v\ge v_{\V_{0}}$ such that $\dom(v)=\V(X)$ and $v[X]\subseteq D$. In virtue of the first assumption, there is a valuation $v^\prime\ge v$ such that $\dom(v^\prime)=\V(X)\cup\V(\alpha)$ and $v^{\prime}[X\cup\{\alpha\}]\subseteq D$. Now, given an arbitrary $a\in|\alg{A}|$, we define a valuation $w$ as follows:
\[
w[q]:=\begin{cases}
	\begin{array}{cl}
		v^\prime[p] &\text{if $p\in\V(X)\cup\V(\alpha)$}\\
		a &\text{otherwise}.
	\end{array}
\end{cases}
\]

It is obvious that $w$ validates $\alpha$ and hence, in virtue of the second assumption, $w$ also validates $\beta$. Since $v_{\V_{0}}\le w$, we conclude that $v_{\V_{0}}\adoptsM\beta$. It remains to apply Lemma~\ref{L:one}.
\end{proof}

The next two definitions were introduced in~\cite{mur2021}.
\begin{defn}[cf.~\cite{mur2021}, definition 3.6]\label{D:weakly-monotonic}
	A relation {\em$\vdash\subseteq\mathcal{P}(\FormsL)\times\FormsL$}, or the operator $\C$ corresponding to this relation, is \textbf{weakly monotonic} if for any set {\em$X\cup Y\cup\{\alpha\}\subseteq\FormsL$} with $X\subseteq Y$ and
	$\Var(X)\cap\Var(\alpha)=\Var(Y)\cap\Var(\alpha)$, $X\vdash\alpha$ implies $Y\vdash\alpha$. 
\end{defn}
\begin{defn}[cf.~\cite{mur2021}, definition 5.1]\label{D:very-strongly-finitary}
	A  relation $\vdash\subseteq\mathcal{P}(\FormsL)\times\FormsL$, or the operator $\C$ corresponding to this relation, is called \textbf{very strongly finitary} if for any nonempty set $X$, if $X\vdash\alpha$, then there is a nonempty $Y\Subset X$ with $\Var(Y)\cap\Var(\alpha)=\Var(X)\cap\Var(\alpha)$ such that $Y\vdash\alpha$.
\end{defn}

The following properties of $\mathcal{M}$-r-consequence were established earlier. 
\begin{prop}[\cite{mur2021}, corollary 3.1]\label{P:Mr-weakly-monotonic}
	Any $\mathcal{M}$-r-consequence is weakly monotonic.
\end{prop}
\begin{prop}[\cite{mur2021}, proposition 5.1]\label{P:Mr-very-strongly-finitary}
If $\mathcal{M}$ is a finite set of finite matrices, then $\mathcal{M}$-r-consequence is very strongly finitary.
\end{prop}

In general, an operator $\C_{\mat{M}}^{r}$ can fail (con-2), and the corresponding relation $\modelMr$ fails the transitivity property:
\[
X\modelMr\beta,~\text{for every $\beta\in Y$, and $Y\cup Z\vdashC\alpha$ imply $X\cup Z\modelMr\alpha$},\tag{\emph{transitivity}}
\]
which is equivalent to the property:
\begin{equation}\label{E:trasitivity-C}
Y\subseteq\C_{\mat{M}}^{r}(X)~\text{implies}~\C_{\mat{M}}^{r}(Y\cup Z)\subseteq\C_{\mat{M}}^{r}(X\cup Z).
\end{equation}

\begin{examp}\label{EX:one}
	Let {\em$\alg{A}=\lr{\textsf{A},\land,\lor,\neg,\one}$} be a nontrivial Boolean algebra with a unit $\one$, and let $\mat{M}=\lr{\alg{A},\{\one\}}$. Then, although $p\modelMr\neg q$, we observe that $p,q\not\modelMr\neg q$; and although $p\modelMr q$ and $q\modelMr\neg p$, we observe that $p\not\modelMr\neg p$.
	
	We note that~\eqref{E:trasitivity-C} is a generalized form of {\em(con-5)}. The property {\em(con-6)} also fails. Indeed, consider inclusion $\{p\}\subseteq\{p,q\}$. Although $\{p,q\}\subseteq\C_{\mat{M}}^{r}(p)$, $\neg q\in\C_{\mat{M}}^{r}(p)\setminus\C_{\mat{M}}^{r}(\{p,q\})$.
	
	In virtue of Proposition~\ref{P:con-connections}-iv, {\em(con-7)} also fails.
\end{examp}

Thus all cumulative properties can fail for some operators $\C_{\mat{M}}^{r}$, even when the matrix $\mat{M}$ is finite, while for all $\mat{M}$, $\C_{\mat{M}}^{r}$ is weakly monotonic, and for all finite $\mat{M}$, $\C_{\mat{M}}^{r}$ is very strongly finitary.

\subsection{M-r$\ast$-consequence and $\mathcal{M}$-r$\ast$-consequence}\label{section:three.three}
Although restricted matrix consequence exhibits interesting properties, there is a problem with it.
\begin{examp}\label{EX:two}
Let $\alg{A}=\lr{\textsf{A},\rightarrow,\one}$ be a nontrivial implicative algebra; cf.~\cite{rasiowa1974}, {\em chapter II}. And let $\mat{M}=\lr{\alg{A},\{\one\}}$. Then  $p\modelMr q$ but $p, q\rightarrow q\not\modelMr q$, although $\modelM q\rightarrow q$ and $p,\alpha\modelMr q$ for any formula $\alpha$ which does not contain $q$.
\end{examp}

 This observation leads to the following consideration.

\vskip 0.1in
\begin{defn}\label{D:M-indepent-variable}
	Let $\mat{M}$ be a matrix and $\alpha$ be a formula. Given $p\in\Var(\alpha)$, $p$ is $\mat{M}$-\textbf{essential} in $($or of$)$ $\alpha$ if there are valuations $v$ and $w$ in $\mat{M}$ such that $v[\alpha]\neq w[\alpha]$, although for any $q\in\V(\alpha)\setminus\{p\}$,
	$v[q]=w[q]$; otherwise $p$ is $\mat{M}$-\textbf{inessential}. Given a family $\mathcal{M}$ of matrices, $p$  is $\mathcal{M}$-\textbf{essential} {\em(}in $\alpha${\em)} if it is $\mat{M}$-essential for some $\mat{M}\in\mathcal{M}$; otherwise $p$ is $\mathcal{M}$-\textbf{inessential}.
\end{defn}

Given a matrix $\mat{M}$ and formula $\alpha$, we denote by $\Vs{\mat{M}}(\alpha)$ the set of all $\mat{M}$-essential variables of $\alpha$; and denote:
\[
\Vs{\mat{M}}(X):=\bigcup_{\alpha\in X}\Vs{\mat{M}}(\alpha).
\]

We state the following two proposition that can be easily verified.
\begin{prop}\label{P:valuations-on-essential-variables}
	Let $\mat{M}$ be a matrix. Given $p\in\Var(\alpha)$, if $p$ is $\mat{M}$-inessential, then
	$v[\alpha]=v[\sigma(\alpha)]$ for any valuation $v$ in $\mat{M}$ and any substitution $\sigma$ satisfying the condition: $\sigma(q)=q$ whenever $q\neq p$.
\end{prop}
\noindent\emph{Proof}~is obvious.

\begin{prop}\label{P:valuations-equal-on-essential}
	Let $\alpha$ be a formula and $v$ and $v^\prime$ be valuations in a matrix $\mat{M}$ such that $v\upharpoonright\Vs{\mat{M}}(\alpha)=v^{\prime}\upharpoonright\Vs{\mat{M}}(\alpha)$. Then $v[\alpha]=v^{\prime}[\alpha]$.
\end{prop}
\noindent\emph{Proof} can be carried out without much efforts by induction on the construction of $\alpha$.

\vskip 0.1in
These two last propositions will be used without reference.

\vskip 0.1in
The concept of an inessential variable induces an additional concept.
\begin{defn}
	Let $\mat{M}$ be a matrix and $a\in|\mat{M}|$. We add a new constant $\mathbf{c}_a$ to the language $\Lan$ along with the agreement that each valuation $v$ in $\mat{M}$, being extended to include $\mathbf{c}_a$ in its domain, satisfies the condition: $v[\mathbf{c}_a]=a$. Now, Given an $\Lan$-formula $\alpha$, any formula that is obtained from the former by replacing all $\mat{M}$-inessential variables of $\alpha$ by $\mathbf{c}_a$ is called a $\mathbf{c}_a$-instance of $\alpha$. If $\alpha$ does not contain $\mat{M}$-inessential variables, a $\mathbf{c}_a$-instance of $\alpha$ is coincident with $\alpha$.
\end{defn}

\begin{prop}
	Let $\mat{M}$ be a matrix, $a,b\in|\mat{M}|$ and $v$ be an arbitrary valuation in $\mat{M}$ extended by the conditions $v[\mathbf{c}_a]=a$ and $v[\mathbf{c}_b]=b$. If $\alpha^{\ast}$ and $\alpha^{\ast\ast}$ are $\mathbf{c}_a$-instance and $\mathbf{c}_b$-instance of $\alpha$, respectively, then $v[\alpha]=v[\alpha^{\ast}]=v[\alpha^{\ast\ast}]$.
\end{prop}
\noindent\emph{Proof}~can be carried out by induction on the complexity of $\alpha$.\\

The last proposition induces the following definition.
\begin{defn}[$\mathbf{c}$-instance]\label{D:c-instance}
	Let $\mat{M}$ be a matrix. An arbitrary fixed $\mathbf{c}_a$-instance of a formula $\alpha$ is called simply a $\mathbf{c}_{\mat{M}}$-instance
	of $\alpha$ is denoted by $\alpha_{\mathbf{c}_{\mat{M}}}$. If $X\subseteq\FormsL$, we denote:
	\[
	X_{\mathbf{c}_{\mat{M}}}:=\set{\alpha_{\mathbf{c}_{\mat{M}}}}{\alpha\in X}.
	\]
\end{defn}

The next equality and the following conditional will be used throughout without reference.
\[
\V(X_{\mathbf{c}_{\mat{M}}})=\V_{\mat{M}}^{\ast}(X)~\text{for every matrix $\mat{M}$}.
\]
\[
Y\subseteq X~\Longrightarrow~\setCM{Y}\subseteq\setCM{X}~\text{for every matrix $\mat{M}$}.
\]

\begin{defn}
	Let $\mat{M}=\lr{\alg{A},D}$ be a matrix. A formula $\alpha$ is called an $\mat{M}$-\textbf{constant} if there is an element $a\in|\alg{A}|$ such that for any valuation $v$ in $\mat{M}$, $v[\alpha]=a$. Given a family $\mathcal{M}$ of matrices, $\alpha$ is an $\mathcal{M}$-\textbf{constant} if it is an $\mat{M}$-constant for each $\mat{M}\in\mathcal{M}$. 
\end{defn}

\begin{prop}\label{P:inessencial-variables}
	Let $\mathcal{M}$ be a family of matrices and $\alpha$ be a formula. If all variables of $\alpha$ are $\mathcal{M}$-inessential, then $\alpha$ is an $\mathcal{M}$-constant.
\end{prop}
\begin{proof}
	Suppose $\mat{M}$ is an arbitrary matrix from $\mathcal{M}$ and $v$ is a valuation in $\mat{M}$. We denote: $a:=v[\alpha]$. Since each variable of $\alpha$ is  $\mat{M}$-inessential, it is not difficult to conclude that for any valuation $v^{\prime}$ in $\mat{M}$, $v^{\prime}[\alpha]=a$. Therefore, $\alpha$ is $\mat{M}$-constant.
\end{proof}

Refining definition~\eqref{E:M-r-consequence}, we arrive at the main definition of this subsection.

Let $\mat{M}=\lr{\alg{A},D}$ and $X\cup\{\alpha\}\subseteq\FormsL$.
\begin{equation}\label{E:M-r*-consequence}
	\begin{array}{rl}
		X\modelMrs\alpha~\define &\text{for any restricted $v$ in $\mat{M}$ with $\dom(v)=\V_{\mat{M}}^{\ast}(X)$},\\
		&\text{there is an extension $w\ge v$ with $\V(X\cup\{\alpha\})\subseteq\dom(w)$}\\
		&\text{such that if $v[X_{\mathbf{c}_{\mat{M}}}]\subseteq D$, then
		$w[\alpha]\in D$}.
	\end{array}
\end{equation}

We note:
\begin{equation}\label{E:default-statement}
	\varnothing\modelMrs\alpha~\Longleftrightarrow~\varnothing\modelMr\alpha.
\end{equation}

Similarly to~\eqref{E:calM-r-consequence}, we define the relation of $\mathcal{M}$-r$^\ast$-consequence.
\begin{equation}\label{E:calM-r*-consequence}
X\modelcalMrs\alpha~\define~X\modelMrs\alpha~\text{for every $\mat{M}\in\mathcal{M}$}.
\end{equation}

The corresponding operators are denoted by $\C_{\mat{M}}^{r\ast}$ and $\C_{\mathcal{M}}^{r\ast}$. Similarly to~\eqref{E:C_calMr}, we have:
\begin{equation}\label{E:C_calMr*}	
\CcalMrs(X)=\bigcap_{\mat{M}\in\mathcal{M}}\CmatMrs(X).
\end{equation}

\vskip 0.1in
First, we focus on M-r$^\ast$-consequence.
From definition~\eqref{E:M-r*-consequence}, there follows the following equivalence.
\begin{equation}\label{E:two-conditions}
	X\modelMrs\alpha~\Longleftrightarrow~X_{\mathbf{c}_{\mat{M}}}\modelMr\alpha.
\end{equation}

\noindent\emph{Proof}~is obvious.\\

We note that~\eqref{E:two-conditions} could be taken as a definition of relation $\modelMrs$.

\begin{prop}\label{P:Mr-stronger-Mr*}
	For any matrix $\mat{M}=\lr{\alg{A},D}$ and a set $X\cup\{\alpha\}\subseteq\FormsL$,
	\[
	X\modelMr\alpha~\Longrightarrow~X\modelMrs\alpha.
	\]
	The converse, in general, does not hold.
\end{prop}
\begin{proof}
	Assume that a valuation $v$ in $\mat{M}$ with $\dom(v)=\Vs{\mat{M}}(X)$ satisfies the condition that $v[X_{\mathbf{c}_{\mat{M}}}]\subseteq D$. We recall that $v[\mathbf{c}_{\mat{M}}]=a$, where
	$a\in|\alg{A}|$. Then, we extend $v$ to a restriction $v^\prime$ with $\dom(v^\prime)=\V(X)$ in such a way that $v^\prime[p]=a$ for any inessential variable $p$ of $X$. It is clear that $v^\prime[X]\subseteq D$. This, by premise, implies that there is an extension $v^{\prime\prime}$ with $\dom(v^{\prime\prime})=X\cup\{\alpha\}$ such that $v^{\prime\prime}[\alpha]\in D$. Thus, we obtain that $X_{\mathbf{c}_{\mat{M}}}\modelMr\alpha$. It remains to apply~\eqref{E:two-conditions}.
	
	Example~\ref{EX:two} shows that the converse can fail.
\end{proof}

\begin{prop}\label{P:Mrs-logical}
Relation $\modelMrs$ is logical relative to $\modelM$.
\end{prop}
\begin{proof}
Indeed, property~\eqref{E:minimum-condition-for-logical-1} follows from Proposition~\ref{P:Mr-logical} and Proposition~\ref{P:Mr-stronger-Mr*}.

To prove property~\eqref{E:minimum-condition-for-logical-2}, we assume that $X\modelMrs\alpha$ and $\alpha\modelM\beta$.
In virtue of~\eqref{E:two-conditions}, $X_{\mathbf{c}_{\mat{M}}}\modelMr\alpha$. According to Proposition~\ref{P:Mr-logical}, $X_{\mathbf{c}_{\mat{M}}}\modelMr\beta$. Then we apply~\eqref{E:two-conditions}.
\end{proof}

We should not expect that relation $\modelMrs$ will enjoy the weak monotonicity property (Definition~\ref{D:weakly-monotonic}). However, we obtain the following.
\begin{prop}\label{P:weak-monotonicity-for-Mrs}
Given a matrix $\mat{M}$, for any set $X\cup Y\cup\{\alpha\}\subseteq\FormsL$ with $X\subseteq Y$ and
$\Vs{M}(X)\cap\V(\alpha)=\Vs{M}(Y)\cap\V(\alpha)$, $X\modelMrs\alpha$ implies $Y\modelMrs\alpha$.
\end{prop}
\begin{proof}
	Let $X\modelMrs\alpha$ and all other premises be fulfilled. Then, in virtue of~\eqref{E:two-conditions}, $\setCM{X}\modelMr\alpha$. Since $\V(\setCM{X})\cap\V(\alpha)=\V(\setCM{Y})\cap\V(\alpha)$
	and $\setCM{X}\subseteq\setCM{Y}$, we can apply the property that $\modelMr$ is weakly monotonic (Proposition~\ref{P:Mr-weakly-monotonic}) in order to conclude that
	$\setCM{Y}\modelMr\alpha$, that is, $Y\modelMrs\alpha$.
\end{proof}

As to finitariness, we obtain the following proposition.
\begin{prop}\label{P:Mrs-strong-finitariness}
Given a matrix $\mat{M}$, for any nonempty $X\subseteq\FormsL$, if $X\modelMrs\alpha$, then there exists a nonempty $Y\Subset X$ such that $\Vs{M}(Y)\cap\V(\alpha)=\Vs{M}(X)\cap\V(\alpha)$ and $Y\modelMrs\alpha$. Moreover, for any set $Z$ with $Y\subseteq Z$ and $\Vs{M}(Z)\cap\V(\alpha)=\Vs{M}(Y)\cap\V(\alpha)$, $Z\modelMrs\alpha$.
\end{prop}
\begin{proof}
	Suppose $X\modelMrs\alpha$, where $X\neq\varnothing$. By~\eqref{E:two-conditions}, $\setCM{X}\modelMr\alpha$.
	In virtue of Proposition~\ref{P:Mr-very-strongly-finitary}, there is a nonempty $X_{0}\Subset\setCM{X}$ such that $\V(X_0)\cap\V(\alpha)=\V(\setCM{X})\cap\V(\alpha)$ and
	$X_{0}\modelMr\alpha$.
	
	It must be clear that for some nonempty $Y\Subset X$, $\setCM{Y}=X_0$. This implies that $\V(\setCM{Y})\cap\V(\alpha)=\V(\setCM{X})\cap\V(\alpha)$ and
	$\setCM{Y}\modelMr\alpha$. Applying~\eqref{E:two-conditions}, we obtain that
	$\Vs{M}(Y)\cap\V(\alpha)=\Vs{M}(X)\cap\V(\alpha)$.
	
	The last part of the statement follows by Proposition~\ref{P:weak-monotonicity-for-Mrs}.
\end{proof}

\vskip 0.1in
Now we will turn to $\mathcal{M}$-r$\ast$-consequence. 
First of all, we observe the following.
\begin{prop}
	Any $\mathcal{M}$-r$\ast$-consequence is logical relative to the corresponding $\mathcal{M}$-consequence.
\end{prop}
\begin{proof}
	To verify (log-1), we notice that, in virtue of Proposition~\ref{P:Mrs-logical}, for every $\mat{M}\in\mathcal{M}$, $\Cn_{\mat{M}}\le\C_{\mat{M}}^{r\ast}$. Then, in virtue 
	of~\eqref{E:C_calMr} and~\eqref{E:C_calMr*}, $\Cn_{\mathcal{M}}\le\CcalMrs$.
	
To verify (log-2), assume that $\alpha\in\CcalMrs(X)$ and $\beta\in\Cn_{\mathcal{M}}(\alpha)$, that is, $X\modelcalMrs\alpha$ and $\alpha\modelcalM\beta$. In virtue of Proposition~\ref{P:Mrs-logical}, for every $\mat{M}\in\mathcal{M}$, $X\modelMrs\beta$, that is, $\beta\in\CcalMrs(X)$.
\end{proof}

\begin{prop}\label{P:weak-monotonicity-for-calMrs}
	If $X\modelcalMrs\alpha$, then for any set $Y$ with $X\subseteq Y$ and $\Vs{\mat{M}}(Y)\cap\V(\alpha)=\Vs{\mat{M}}(X)\cap\V(\alpha)$ for every $\mat{M}\in\mathcal{M}$, $Y\modelcalMrs\alpha$.
\end{prop}
\noindent\textit{Proof}~immediately follows from Proposition~\ref{P:weak-monotonicity-for-Mrs}.\\ 

\begin{prop}\label{P:calMrs-strongly-finitariness}
	Let $\mathcal{M}$ be a finite family of finite matrices.
	For any nonempty $X\subseteq\FormsL$, if $X\modelcalMrs\alpha$, then there exists a nonempty $Y\Subset X$ such that $Y\modelcalMrs\alpha$ and $\Vs{\mat{M}}(Y)\cap\V(\alpha)=\Vs{\mat{M}}(X)\cap\V(\alpha)$ for every $\mat{M}\in\mathcal{M}$. Moreover, for any set $Z$ with $Y\subseteq Z$ and $\Vs{\mat{M}}(Z)\cap\V(\alpha)=\Vs{\mat{M}}(Y)\cap\V(\alpha)$ for every $\mat{M}\in\mathcal{M}$, $Z\modelcalMrs\alpha$.
\end{prop}
\begin{proof}
	Suppose $\mathcal{M}=\{\mat{M}_1,\ldots,\mat{M}_n\}$ and $X\modelcalMrs\alpha$, that is,
	\[
	X\models^{r\ast}_{\mat{M}_i}\alpha~\text{for each $i\in\{1,\ldots,n\}$},
	\]
	where $X\neq\varnothing$.
	
	According to Proposition~\ref{P:Mrs-strong-finitariness}, for each $i$, there is a nonempty $Y_i\Subset  X$ such that $\V^{\ast}_{\mat{M}_i}(Y_i)\cap\{\alpha\}=\V^{\ast}_{\mat{M}_i}(X)\cap\V(\alpha)$. We also note that $\V^{\ast}_{\mat{M}_i}(Y_j)\cap\{\alpha\}\subseteq\V^{\ast}_{\mat{M}_i}(X)\cap\V(\alpha)$ for any $i,j\in\{1,\ldots,n\}$.
	
	Next we denote:
	\[
	Y:=Y_1\cup\ldots\cup Y_n.
	\]
	
	We observe that:
	\begin{itemize}
		\item $Y\neq\varnothing$ and $Y\Subset X$;
		\item $\V^{\ast}_{\mat{M}_i}(Y)\cap\{\alpha\}=\V^{\ast}_{\mat{M}_i}(X)\cap\V(\alpha)$ for each $i\in\{1,\ldots,n\}$.
	\end{itemize}
	Applying Proposition~\ref{P:weak-monotonicity-for-Mrs}, we conclude that 
	\[
	Y\models^{r\ast}_{\mat{M}_i}\alpha~\text{for each $i\in\{1,\ldots,n\}$,}
	\]
	that is, $Y\modelcalMrs\alpha$.
	
	Now if for each $i\in\{1,\ldots,n\}$, for some $Z$ with
	$Y\subseteq Z$ and $\Vs{\mat{M}_i}(Z)\cap\V(\alpha)=\Vs{\mat{M}_i}(Y)\cap\V(\alpha)$, then, according to Proposition~\ref{P:weak-monotonicity-for-calMrs}, $Z\models^{r\ast}_{\mat{M}_i}\alpha$ for each $\mat{M}_{i}\in\mathcal{M}$, that is, $Z\modelcalMrs\alpha$.
\end{proof}

\begin{cor}
	Let $\mathcal{M}$ be a finite family of finite matrices.
	Then $\mathcal{M}$-r$\ast$-consequence is strongly finitary.
\end{cor}

\section{Substitution}
According to Proposition~\ref{P:lukasiewicz-tarski}, any $\Cn_{\mat{M}}$ operator is structural. In terms of consequence relations, this means that, given a matrix $\mat{M}$, for any substitution $\sigma$,
\[
X\modelM\alpha~\Longrightarrow~\sigma(X)\modelM\sigma(\alpha).
\]

The structurality property can hardly be expected for $\modelMr$ or $\modelMrs$ relations. However, the following rule can be considered for these relations, although not without some restriction.
\begin{equation}\label{E:substitution:one}
X\modelMr\sigma(\alpha)~\Longrightarrow~X\modelMr\alpha,
\end{equation}
providing that $\V(X)\cap\V(\alpha)=\varnothing$. (In Section~\ref{section:five-two} this property is formulated as the \textit{rule of reverse substitution}.)

It is obvious that~\eqref{E:substitution:one} is true.
Indeed, the restriction `$\V(X)\cap\V(\alpha)=\varnothing$' allows us to reduce~\eqref{E:substitution:one} to the implication: for any valuation $v$ in $\mat{M}$, if $v$ validates $\sigma(\alpha)$, then $v\circ\sigma$ validates $\alpha$. On the other hand, if we drop the restriction, we face with the problem that when a valuation $v$ validates $X$ and its extension $w$ validates $\sigma(\alpha)$, the valuation $w\circ\sigma$, which validates $\alpha$, is not necessarily an extension of $v$.

\vskip 0.1in
It is also true that~\eqref{E:substitution:one} can be extended to the following:
\[
X\modelcalMr\sigma(\alpha)~\Longrightarrow~X\modelcalMr\alpha
\]
and
\[
X\modelcalMrs\sigma(\alpha)~\Longrightarrow~X\modelcalMrs\alpha,
\]
providing that $\V(X)\cap\V(\alpha)=\varnothing$.

The last implication can obviously be reduced to the case when $\mathcal{M}$ consists of a single matrix:
\begin{equation*}\label{E:substitution:two}
	X\modelMrs\sigma(\alpha)~\Longrightarrow~X\modelMrs\alpha,
\end{equation*}
where $\V(X)\cap\V(\alpha)=\varnothing$.

Indeed, according to~\eqref{E:two-conditions}, the last implication is equivalent to the following:
\[
X_{\mathbf{c}_{\mat{M}}}\modelMr\sigma(\alpha)~\Longleftrightarrow~X_{\mathbf{c}_{\mat{M}}}\modelMr\alpha.
\]
It remains to notice that $\V(X_{\mathbf{c}_{\mat{M}}})\cap\V(\alpha)=\varnothing$ whenever $\V(X)\cap\V(\alpha)=\varnothing$.

\section{Logical friendliness relation and beyond}\label{section:four}
So far, the specifics of objective language have not been used in our discussion to achieve positive results, except that the language was sentential. But to get negative results, we had to turn to specific languages.  However, specifying an objective language can also open up more possibilities for defining nonmonotonic consequence relations.

For the rest of this paper, we will use the language of Example~\ref{EX:one}, that is a language with the logical connectives $\land$, $\lor$, $\rightarrow$, $\neg$ and the constant $\top$. We denote this formal language by $\LanF$. Unspecified formulas of $\LanF$ we denote by letters $A$, $B$, etc., and sets of such formulas by letters $\Gamma$, $\Delta$, etc. We continue using the same notation for substitutions and valuations. Since all our matrices will be Heyting algebras with a greatest element $\one$, for each valuation $v$, we require the condition: $v[\top]=\one$.

\subsection{Logical friendliness} 
Our interest to nonmonotonic relation was inspired by~\cite{makinson2005a} and~\cite{makinson2007}, where the relation of \emph{logical friendliness} was introduced.\footnote{Makinson has proved that logical friendliness is strongly finitary with some additional property for monotonicity, which is however weaker than the one formulated in Proposition 5.1 of~\cite{mur2021}.} To define logical friendliness, we use the matrix $\booleTwo$ which is a two-element Boolean algebra with the filter $\{\one\}$; namely
\begin{equation}\label{E:friendliness-definition}
\Gamma\snakeF A~\define~\Gamma\models_{\booleTwo}^{r}A.
\end{equation}

Further, we define:
\begin{equation}\label{E:friendliness-restricted}
\Gamma\snakeFs A~\define~\Gamma\models_{\booleTwo}^{r\ast}A.
\end{equation}

We denote the operators corresponding to $\snakeF$ and $\snakeFs$ by $\CF$ and $\CFs$, respectively.

In view of Example~\ref{EX:one} and Example~\ref{EX:two}, we observe:
\begin{equation}
	\Cn_{\booleTwo}<\CF<\CFs.
\end{equation}

The following proposition later will play the role of a navigator, but it is convenient to discuss this property here. It reads that for defining logical friendliness, we can use any nontrivial Boolean algebra.
\begin{prop}\label{P:lemma-novigator}
Let $\alg{B}$ be a nontrivial Boolean algebra that is in combination with a logical filter $\{\one\}$ forms a logical matrix $\mat{B}$. Then
\begin{equation}\label{E:Boolean=friendlines}
	\Gamma\models_{\mat{B}}^{r}A~\Longleftrightarrow~\Gamma
	\models_{\booleTwo}^{r}A.
\end{equation}
\end{prop}
\begin{proof}
	We recall that $\alg{B}$ (understood as an algebra, not matrix) is a subdirect product $\booleTwo^{I}$ (the Cartesian product of copies of $\booleTwo$), where $\booleTwo$ is also regarded here as an algebra. This means that $\alg{B}$ is embedded to $\booleTwo^{I}$ so that each projection, $g_i$, is an epimorphism. We also recall that each valuation in an algebra is a homomorphism from the corresponding formula algebra to the former.
	
	Now assume that $\Gamma\models_{\mat{B}}^{r}A$ and consider a valuation $v$ in $\booleTwo$ such that $v[\Gamma]\subseteq\{\one\}$ and $\dom(v)=\V(\Gamma)$.
	
	We define a valuation $\hat{v}$ in $\alg{B}$ as follows.
	\[
	\hat{v}[p]:=\begin{cases}
		\begin{array}{cl}
			\zer &\text{if $v[p]=\zer$},\\
			\one &\text{if $v[p]=\one$},~\text{for every $p\in\V(\Gamma)$}.
		\end{array}
	\end{cases}
	\]
	
	It is clear that $\hat{v}[\Gamma]\subseteq\{(\one)_{I}\}$. Therefore, by premise, there is an extension $w\ge\hat{v}$ such that
	$w[\Gamma\cup\{\alpha\}]\subseteq\{(\one)_{I}\}$. Now, $v^\ast:=g_{i}\circ w$ is a valuation in $\booleTwo$ such that $v^\ast\ge v$ and $v^\ast[\Gamma\cup\{A\}]\subseteq\{\one\}$.
	
	Next assume that $\Gamma
\models_{\booleTwo}A$ and that a valuation $w$ in $\alg{B}$ validates $\Gamma$, that is, $w[\Gamma]\subseteq\{(\one)_{I}\}$. Then, obviously, valuation $v:=g_{i}\circ w$ validates $\Gamma$ in $\booleTwo$, that is, $v[\Gamma]\subseteq\{\one\}$. Then there is an extension $v^{\prime}\ge v$ such that $v^{\prime}[\Gamma\cup\{A\}]\subseteq\{\one\}$. Now we define a valuation $w^\prime$ in $\alg{B}$ (regarded as a subalgebra of $\booleTwo^I$) according to the rule: $(w^{\prime}[p])_{i}:=v^{\prime}[p]$. It is clear that $w^{\prime}\ge w$ and $w^{\prime}[\Gamma\cup\{A\}]\subseteq\{(\one)_{I}\}$.
\end{proof}

\subsection{Beyond logical friendliness}\label{section:five-two}
We aim to make logical friendliness the starting point for getting more logical nonmonotonic consequences. This time we choose the method of proof theory.

Let $\Int$ and $\Cl$ denote the classes of intuitionistic and classical propositional tautologies, respectively. An \emph{intermediate logic} is a set $L$ of formulas, satisfying the following conditions: $\Int\subseteq L\subseteq\Cl$, $L$ is closed under uniform substitution and under modus ponens. The class of all intermediate logics is denoted by $\ExtInt$.

Let $L\in\ExtInt$. A finite sequence
\[
A_1, A_2, \ldots, A_n
\]
is called an $L$-\emph{derivation from a set} $\Gamma$ if for each $A_i$, either $A_i\in L$ or $A_i\in\Gamma$ or $A_i$ can be obtained by modus ponens from two preceding formulas of the sequence. In addition, we say that the sequence is an $L$-derivation of the last formula of the sequence, that is, of $A_n$.

Next, we define:
\[
\Gamma\vdashL{L}A~\define~\text{there is an $L$-derivation of $A$ from $\Gamma$}.
\]

Each relation $\vdashL{L}$ obviously is a monotonic consequence relations. We denote the corresponding consequence operator by $\Cn_L$.

We call a set $\Gamma$ a $\Cn_L$-\emph{theory} if $\Cn_L(\Gamma)=\Gamma$. (This concept will be employed later.)

\vskip 0.1in
We call any expression of the form $\Gamma\sequentL{}A$ a \textit{sequent}. 

Next, we provide the list of \textit{sequential} $L$-\textit{axioms} (or \textit{sL}-\emph{axioms} for short) and that of \textit{sequential} $L$-\textit{rules} (or \textit{sL}-\textit{rules} for short).

\begin{itemize}
	\item[\underline{s\textit{L}-\text{Axioms}}:]
	\item[axiom 1:] $\sequentL{}p$, for any $p\in\VarF$;\\
	\item[axiom 2:]  $\Gamma\sequentL{}\top$;\\
	\item[axiom 3:]  $\Gamma\sequentL{}\bigwedge\Delta$ if $\Delta\Subset\Gamma$ ($\bigwedge\Delta:=\top$ if $\Delta=\varnothing$);\\
	\item[axiom 4:] $\Gamma\sequentL{}A$ whenever $\Gamma\vdash_{L} A$.\\\\ 
	
	\item[\underline{s\textit{L}-\text{Rules}}:]
	\item[rule 1:] $\dfrac{\sequentL{}A}{\Gamma\sequentL{}A}$, providing that
	$\Var(\Gamma)\cap\Var(A)=\varnothing$;\\\\
	\item[rule 2:] $\dfrac{\Gamma, A\sequentL{}C
		~\text{and}~\Delta, B\sequentL{}C}{\Gamma\cup\Delta, A\lor B\sequentL{}C}$,
	providing that $\Var(\Gamma\cup\{A\})=\Var(\Delta\cup\{B\})$;\\\\
	\item[rule 3:]
	$
	\dfrac{\Gamma\sequentL{}\sigma(A)}{\Gamma\sequentL{}A}, 
	$
	providing that $\Var(\Gamma)\cap\Var(A)=\varnothing$; (Reverse substitution)\\\\
	\item[rule 4:]
	$
	\dfrac{\Gamma\sequentL{}A,~\text{and}~A\sequentL{} B}{\Gamma\sequentL{}B}, 
	$
	providing that either $\Var(\Gamma)\subseteq\Var(A)$ or\\ $\Var(\Gamma)\cap\Var(B)\subseteq\Var(A)\subseteq\Var(\Gamma)$; (Cut)\\\\
	\item[rule 5:]
	$
	\dfrac{\Gamma, A\sequentL{}B~\text{and}~C\sequentL{} A}{\Gamma, C\sequentL{}B},
	$
	providing that $\Var(C)\subseteq\Var(\Gamma\cup\{A\})$;\\ (Deductive replacement in antecedent)\\
	\item[rule 6:]
	$
	\dfrac{\Gamma\sequentL{}A~\text{and}~A\sequentL{} B}{\Gamma\sequentL{}B}$. (Deductive replacement in consequent)
\end{itemize}

A finite nonempty list
\[
\Gamma_1\sequentL{}A_1,\ldots,\Gamma_n\sequentL{}A_n
\]
is called an \textit{sL}-\textit{derivation} (of the last sequent on the list, i. $\Gamma_n\sequentL{}A_n$) if each $\Gamma_k\sequentL{}A_k$ is either a s\textit{L}-axiom or can be obtained from preceding sequents by one of the rules 1--6. We say that 
$\Gamma\sequentL{}A$ \textit{holds} in s$L$-calculus
if there is an s\textit{L}-derivation ending with $\Gamma\sequentL{}A$. 

This leads to the following definition: for each $l\in\ExtInt$,
\begin{equation}\label{E:sequentL-definition}
	\Gamma\VdashL{L}A~\define~\Gamma\sequentL{L}A~\text{holds in s$l$-calculus}.
\end{equation}

\vskip 0.1in
The motivation for formulating the s$L$-calculi is based on the following observation.
\[
\Gamma\snakeF A~\Longleftrightarrow~\Gamma\VdashL{\Cl}A~(\text{cf.~\cite{muravitsky2009}, theorems 4.5 and 4.8}).
\]

Reformulating the last equivalence, we obtain completeness theorem for the nonmonotonic consequence $\VdashL{\Cl}$:
\[
\Gamma\VdashL{\Cl}A~\Longleftrightarrow~\Gamma\models_{\booleTwo}^{r}A.
\]

This suggests that there may be similar completeness theorems for other $\VdashL{L}$ relations. We will come back to this later, but now we will discuss what is interesting about the $\VdashL{L}$ relations.

According to~~\cite{mur2021}, proposition 7.3, each relation $\VdashL{L}$ is reflexive and nonmonotonic. Further, for any nonempty set $\Gamma$ and any formula $A$ with $\V(\Gamma)\cap\V(A)=\varnothing$, if $\Gamma\VdashL{L}A$, then there is a nonempty set $\Delta\Subset\Gamma$ such that
$\Delta\VdashL{L}A$. To this, we add the following.
\begin{prop}
	Each $\VdashL{L}$ is logical relative to $\vdashL{L}$.
\end{prop}
\begin{proof}
Indeed, suppose $\Gamma\vdashL{L}A$. Then, by axiom 4, $\Gamma\sequentL{L}A$ holds, that is, $\Gamma\VdashL{L}A$.

Next, assume that $\Gamma\VdashL{L}A$ and $A\vdashL{L}B$. The first assumption implies that there is an s$L$-derivation of the s$L$-sequent $\Gamma\sequentL{L}A$.
Then we apply rule 6. Thus there is an s$L$-derivation of the s$L$-sequent $\Gamma\sequentL{L}B$, that is, $\Gamma\VdashL{L}B$.
\end{proof}

Now let us return to the question of the possibility of the completeness theorem for each relation $\VdashL{L}$.

We recall that, in the terminology of \cite{citkin-muravitsky2022}, each $\vdashL{L}$ is a \emph{unital assertional abstract logic}.\footnote{For definitions, the reader is offered to consult~\cite{citkin-muravitsky2022}, chapter 6.} 

For a fixed $L\in\ExtInt$, we denote by $\Sigma_{L}$ the set of all $L$-theories. Then, given $D\in\Sigma_{L}$, we denote by $\LT{L}[D]$ the \emph{Lindenbaum-Tarski algebra relative to} $D$. Further, each $\LT{L}[D]$ is a homomorphic image of $\LT{L}[\Cn_{L}(\varnothing)]$ and, hence, is a Heyting algebra. Moreover, each such algebra can be regarded as a unital matrix when the logical filter consists of one element. In each $\LT{L}[D]$, the designated element is the unit of the algebra, that is, its greatest element.

In the case $L=\Cl$, each $\LT{\Cl}[D]$ is a Boolean algebra.
This circumstance will play an important role soon.

By proposition 6.3.5 of~\cite{citkin-muravitsky2022},
\begin{equation}
\Gamma\vdashL{L}A~\Longleftrightarrow~\Gamma\models_{\LT{L}[D]}A~\text{for every $D\in\Sigma_{L}$}.
\end{equation}

This allows us to say that each $\vdashL{L}$ is the $\mathcal{M}_{L}$-consequence, where \[\mathcal{M}_{L}:=\set{\LT{L}[D]}{D\in\Sigma_{L}}.
\]

 In the case $L=\Cl$, we have:
\[
\Gamma\vdashL{\Cl}A~\Longleftrightarrow~\Gamma\models_{\LT{\Cl}[D]}A~\text{for every $D\in\Sigma_{\Cl}$}.
\]

Now, when we move from $\mathcal{M}_{\Cl}$ consequence to 
$\mathcal{M}_{\Cl}$-r-consequence, in virtue of Proposition~\ref{P:lemma-novigator}, we obtain the conclusion: relation $\snakeF$ is the $\mathcal{M}_{\Cl}$-r-consequence. This suggests the following.

\begin{conjecture}
Each $\VdashL{L}$ is the $\mathcal{M}_{L}$-r-consequence.
\end{conjecture}

Of course, from $\mathcal{M}_{L}$-r-consequence, we can move to $\mathcal{M}_{L}$-r$\ast$-consequence. Then, we define:
\[
\Gamma\VdashL{L^\ast}A~\define~\Gamma_{\mathbf{c}}\VdashL{L}A,
\]
where
\[
\Gamma_{\mathbf{c}}:=\set{B_{\mathbf{c}_{\LT{L}[X]}}}{B\in\Gamma}
\]
and each $B_{\mathbf{c}_{\LT{L}[X]}}$ is obtained from $B$ by replacement of each occurrence of a variable inessential in $\LT{L}[\Gamma]$ with $\top$. (We recall that each $\LT{L}[\Gamma]$ is a Heyting algebra.)

The following conjecture suggests itself.
\begin{conjecture}
	$\Gamma\VdashL{L^\ast}A$ if, and only if, $\Gamma\modelMrs A$ for every $\mat{M}\in\mathcal{M}_{L}$.
\end{conjecture}

\section{What has been achieved?}
We started our discussion about the concept of nonmonotonic reasoning from a philosophical point of view. Popular belief says that time can render the final verdict.
But what is meant by saying this? We can suggest two possible answers.

The first is contained in the following quote from R. von Mises.

\vskip 0.1in
\begin{quote}
	I aim at the construction of a rational theory, based on the simplest possible concepts, one which, although admittedly inadequate to represent the complexity of the real processes, is able to reproduce satisfactorily some of their essential properties. (\cite{mises1981}, The inadequacy of theories)
\end{quote}

\vskip 0.1in
The second opinion is more difficult to express. 

The success of the proposed approach is often provided by unexplored space for its further development. The larger this space, the more questions left unanswered, the more tools available to answer these questions, the more likely it is that this approach will become the main focus of research. Such an approach could be termed \emph{open}, that is, open for further development.

This view of mine echoes another statement by von Mises.

\vskip 0.1in
\begin{quote}
	$[\ldots]$ the value of a concept is not gauged by its correspondence with some usual group of notions, but only by its usefulness for further scientific development $[\ldots]$ (\cite{mises1981}, Synthetic definitions)
\end{quote}

To be open in the above sense means not to be sterile in the understanding of G. Kreisel, who wrote,

\vskip 0.1in
\begin{quote}
Aristotle gives much attention to the matter of \emph{choice}
of abstractions. This indicates a shift of emphasis away from matters of principle (e.g. mere validity of analogies) to a focus (on, of course, valid cases) which provides safeguards against \emph{sterility} (not merely against straight error). (\cite{kreisel-1}, Choice of abstractions)
\end{quote}

\vskip 0.1in
Whether a (consistent) approach is open enough to avoid sterility can only be judged by time.

\vskip 0.1in
In a broader context, concerning the \textit{what} and \textit{how} questions, the difficulty of the task should not overshadow the more important level associated with the act of knowing. This thought is not mine, but a paraphrase of the thought of Thomas Aquinas, who wrote,
\begin{quote}
	When something is more difficult, it is not for that reason necessarily more worthwhile, but it must be more difficult in such a way, as also to be at a higher level of goodness. [Read: ``the act of knowing'' instead of ``goodness''] (Quoted from~\cite{pieper2009}, chapter II.)
\end{quote}


\end{document}